\documentclass[a4paper]{article}
\usepackage[latin1]{inputenc}
\usepackage{amsmath}
\usepackage{amsfonts}
\usepackage{amssymb}
\usepackage{graphicx}
\usepackage{amsthm} 	
\usepackage{hyperref}
\usepackage{todonotes}
\usetikzlibrary{backgrounds}	
\usepackage[toc]{appendix}

\usepackage{graphicx,xcolor}	

\usepackage{soul}

\usepackage{xcolor}
\pagecolor{white}

\usepackage{esint}				

\usepackage[a4paper,top=1.7cm,bottom=1.7cm,left=1.7cm,right=1.7cm]{geometry} 

\allowdisplaybreaks

\newtheorem{theorem}{Theorem}[section]
\newtheorem{proposition}[theorem]{Proposition}
\newtheorem{corollary}[theorem]{Corollary}

\title{A degenerating Robin-type traction problem in a periodic domain}
\author{ Matteo Dalla Riva  \thanks{Dipartimento di Ingegneria, Universit\`a degli Studi di Palermo, Viale delle Scienze, Ed. 8, 90128 Palermo, Italy. E-mail: matteo.dallariva@unipa.it } ,  Gennady Mishuris \thanks{Department of Mathematics, Aberystwyth University, Aberystwyth, SY23 3BZ Wales, UK. E-mail: ggm@aber.ac.uk} , Paolo Musolino\thanks{Dipartimento di Scienze Molecolari e Nanosistemi, Universit\`a Ca' Foscari Venezia, via Torino 155, 30172 Venezia Mestre, Italy. E-mail: paolo.musolino@unive.it}\ \thanks{{\it Corresponding author.}}}

\date{September 06, 2022}

\begin{document}

\maketitle

\noindent

{\small
{\bf Abstract:}  We consider a linearly elastic material with a periodic set of voids. On the boundaries of the voids we set a Robin-type traction condition. Then we investigate the asymptotic behavior of the displacement solution as the Robin condition turns into a pure traction one.  To wit, there will be a matrix function {$b[k](\cdot)$ that depends analytically on a real parameter $k$ and vanishes for $k=0$ and we multiply the Dirichlet-like part of the Robin condition by $b[k](\cdot)$}. We show that the displacement solution can be written in terms of power series of $k$ that converge for $k$ in a whole neighborhood of $0$. For our analysis we use  the Functional Analytic Approach.

\vspace{9pt}

\vspace{9pt}

}

\section{Introduction}\label{introd}

There is almost a century of literature on the mathematical analysis of  perforated plates and porous materials
(see, e.g., \cite{Na35, Ho35, Ho52, BaHi60, GoJa65}, see also Mityushev {\it et al.} \cite{MiAnGl22} for a review). {Examples can be found in  the applications to the study of porous coatings and  of interfacial coatings. If the thickness of the coating is  smaller than the characteristic size of the pores  (while its material properties are  weaker/softer  in comparison with those of the main composite material), the corresponding mathematical problems may degenerate into periodic  problems with conditions of Robin type (see \cite{KlMo98, AnAvKoMo01, Mi01, Mi04, SoPiMiMi15, XuTiXi20}, see also the books by Milton \cite[Chap.~1]{Mi02} and Movchan {\it et al.}~\cite{MoMoPo02}).} In dimension two, these problems can be analyzed with complex variable techniques  (see, e.g., Kapanadze {\it et al.~}\cite{KaMiPe15b},
Dryga\'s {\it et al.~}\cite{DrGlMiNa20},
Gluzman {\it et al.~}\cite{GlMiNa18}). For analog problems in dimension $n \geq 2$, one may resort to integral equation methods (as, {e.g.,}  in Ammari and Kang \cite{AmKa07}).

In this paper we study the Lam\'e system in a periodic domain with a Robin-type traction condition on the boundary. The trace of the displacement part (we may also say the ``Dirichlet-type term") of the boundary condition is multiplied by a matrix function ${b[k](\cdot)}$ that depends analytically on a positive parameter $k$ and that vanishes for $k=0$. Then, as $k$ approaches $0$ the Robin boundary condition degenerates into a pure traction one (a natural condition to have when dealing with the Lam\'e equations). We  study the map that takes $k$ to the displacement solution $u[k]$ and, under suitable conditions on ${b[k](\cdot)}$, we show that $k\mapsto u[k]$ can be described in terms of power series of $k$ that converge for $k$ in a neighborhood of $0$.  A similar result was obtained in \cite{MuMi18} for the analog problem in a bounded domain with a single hole. Here we show that the approach of \cite{MuMi18} can be adapted to the case of infinite periodic domains.

\subsection{The problem}

We start by presenting the geometric setting. We fix once for all
\[
n\in {\mathbb{N}}\setminus\{0,1 \}\,,\qquad  (q_{11},\dots,q_{nn})\in]0,+\infty[^{n}\,.
\]
Here ${\mathbb{N}}$ denotes the
set of natural numbers including $0$. We take  $Q :=\Pi_{j=1}^{n}]0,q_{jj}[$ as fundamental periodicity cell and we denote by $q$ the diagonal matrix with $(j,j)$ entry equal to $q_{jj}$ for all $j \in \{1,\dots,n\}$. We construct our periodic domain by removing from $\mathbb{R}^n$ congruent copies of a bounded domain of class $C^{m,\alpha}$. (For the definition of sets and functions of the Schauder class $C^{j,\alpha}$ ($j \in \mathbb{N}$)  we refer, e.g., to Gilbarg and
Trudinger~\cite{GiTr83}.) Therefore, we fix once and for all $m\in {\mathbb{N}}\setminus\{0\}$, $\alpha\in]0,1[$, and
we assume that
\[
\text{${\Omega_Q}$ is a bounded open subset of ${\mathbb{R}}^{n}$ of class $C^{m,\alpha}$ such that $\overline{\Omega_Q}\subseteq Q$.}
\]
We define the periodic domain
\[
{\mathbb{S}} [\Omega_Q]^{-}{:=} {\mathbb{R}}^{n}\setminus\bigcup_{z\in{\mathbb{Z}}^{n} }(qz+\overline{\Omega_Q})\, .
\]
{To introduce the Lam\'e equations in ${\mathbb{S}} [\Omega_Q]^{-}$, we denote by $T$ the function from $ ]1-(2/n),+\infty[\times M_n(\mathbb{R})$ to $M_n(\mathbb{R})$ defined by $T(\omega,A){:=} (\omega-1)(\mathrm{tr}A)I_n+(A+A^t)$  for all  $\omega \in ]1-(2/n),+\infty[$, $A \in M_n(\mathbb{R})$.} Here $M_n(\mathbb{R})$ denotes the space of $n\times n$ matrices with real entries, $I_n$ denotes the $n\times n$ identity matrix, $\mathrm{tr}A$ and $A^t$ denote the trace and the transpose matrix of $A$, respectively. {We note} that if we set $L[\omega]{:=} \Delta+\omega \nabla \mathrm{div}$, then $L[\omega]u=\mathrm{div} \,T(\omega,Du)$ for all regular vector valued functions $u$, where  $Du$ denotes the Jacobian matrix of  $u$. Then we take
\[
\text{a matrix $B \in M_n(\mathbb{R})$, a function $g \in C^{m-1,\alpha}(\partial \Omega_Q, \mathbb{R}^n)$, and a function $b \in C^{m-1,\alpha}(\partial \Omega_Q,M_n(\mathbb{R}))$}
\]
such that
\begin{align*}
&\text{$\bullet$ {$\xi^t b(x)\xi \leq 0$ for all $\xi \in \mathbb{R}^n$ and all $x\in\partial\Omega_Q$},}\\
&\text{$\bullet$ $\mathrm{det}\int_{\partial \Omega_Q}b\, d\sigma\neq 0$.}
\end{align*}
(Note that the last condition implies that $\mathrm{det}\,b(x_0)\neq 0$ for some {$x_0\in\partial\Omega_Q$}.)
With these ingredients we   {write a boundary} value problem for the Lam\'e equation in ${\mathbb{S}} [\Omega_Q]^-$ with a Robin-type boundary condition. That is,
\[
 \left \lbrace
 \begin{array}{ll}
  \mathrm{div}\, T(\omega, Du )= 0 & \mathrm{in}\  {\mathbb{S}} [\Omega_Q]^-\,, \\
u(x+qe_j) =u(x) +Be_j&  \textrm{$\forall x \in \overline{\mathbb{S}[\Omega_Q]^{-}}, \forall j\in \{1,\dots,n\}$}, \\
T(\omega,Du(x))\nu_{\Omega_Q}(x)+ b(x)u(x)=g(x) & \textrm{$\forall x \in \partial \Omega_Q$}\,,
 \end{array}
 \right.
\]
where $\{e_1,\dots,e_n\}$ is the canonical basis of $\mathbb{R}^n$ and $\nu_{\Omega_Q}$ denotes the outward unit normal to $\partial \Omega_Q$. We {know} that for  $\omega \in ]1-(2/n),+\infty[$ the solution $u$ of the problem above exists, is unique, and belongs to the Schauder class $C^{m,\alpha}(\overline{\mathbb{S}[\Omega_Q]^{-}}, \mathbb{R}^n)$  (see \cite[Thm.~4.4]{DaMiMu22}). We now want the term $b(x)u(x)$ in the boundary condition to disappear in a suitable way as a certain positive parameter $k$ tends to zero. Then we take $k_0>0$ and we introduce
\[
\text{an analytic map $]-k_0,k_0[\ni k\mapsto b[k]\in C^{m-1,\alpha}(\partial \Omega_Q,M_n(\mathbb{R}))$}
\]
that satisfies the following conditions:
\begin{align}
&\text{$\bullet$ {$\xi^t b[k](x)\xi \leq 0$ for all $\xi \in \mathbb{R}^n$, all $x\in\partial\Omega_Q$, and all $k\in]0,k_0[$}}\label{c1}\\
&\text{$\bullet$ $\mathrm{det}\int_{\partial \Omega_Q}b[k]\, d\sigma\neq 0$ for all $k\in ]0,k_0[$,}\label{c2} \\
&\text{$\bullet$ $b[0]=\lim_{k\to 0}b[k]=0$,\label{c3}}
\end{align}
and for a fixed $\omega \in ]1-(2/n),+\infty[$    we consider the following problem
 \begin{equation}\label{bvp}
 \left \lbrace
 \begin{array}{ll}
  \mathrm{div}\, T(\omega, Du )= 0 & \mathrm{in}\  {\mathbb{S}} [\Omega_Q]^-\,, \\
u(x+qe_j) =u(x) +Be_j&  \textrm{$\forall x \in \overline{\mathbb{S}[\Omega_Q]^{-}}, \forall j\in \{1,\dots,n\}$}, \\
T(\omega,Du(x))\nu_{\Omega_Q}(x)+ b[k](x)u(x)=g(x) & \textrm{$\forall x \in \partial \Omega_Q$}\,.
 \end{array}
 \right.
 \end{equation}
 For each $k\in ]0,k_0[$  problem \eqref{bvp} has a unique {solution $u\in C^{m,\alpha}(\overline{\mathbb{S}[\Omega_Q]^{-}},\mathbb{R}^n)$},  which we denote by $u[k]$ to emphasize its dependence on $k$. For $k=0$ we have $b[0]=0$ and the Robin-type traction condition of problem \eqref{bvp} degenerates into Neumann-type one. The resulting pure traction problem may be not solvable if certain compatibility conditions are not satisfied.

\subsection{The main result}

 Our aim is to study the asymptotic behavior of the solution $u[k]$ of problem \eqref{bvp} as $k> 0$ approaches zero. More precisely, we plan to apply the {\it Functional Analytic Approach} of \cite{DaLaMu21} and the periodic elastic single layer potential $v_{q}^-[\omega,\cdot]$ (see Section \ref{notprel}) to represent the solution in terms of convergent power series in suitable Banach spaces. To succeed, however, we need an additional assumption on the function {$k\mapsto b[k]$}. By conditions \eqref{c1} and \eqref{c3} and by the real analyticity of the function $k\mapsto b[k]$ we see that there exists $l\in\mathbb{N}\setminus\{0\}$ such that
\begin{align*}
&\text{$\bullet$ $k\mapsto k^{-l}b[k]$ is real analytic from $]-k_0,k_0[$ to $C^{m-1,\alpha}(\partial \Omega_Q,M_n(\mathbb{R}))$,}\\
&\text{$\bullet$ the matrix function $\tilde b:=\lim_{k\to 0}k^{-l}b[k]$ belongs to $C^{m-1,\alpha}(\partial \Omega_Q,M_n(\mathbb{R}))$ and is not $0$,}\\
&\text{$\bullet$ {$\xi^t \tilde{b}(x)\xi \leq 0$ for all $\xi \in \mathbb{R}^n$ and all $x\in\partial\Omega_Q$.} }
\end{align*}
 We shall further assume that
 \begin{equation}\label{c4}
 \text{$\bullet$ we have $\mathrm{det}\int_{\partial \Omega_Q}\tilde b\, d\sigma\neq 0$.}
 \end{equation}
Then, with conditions \eqref{c1}, \eqref{c2}, \eqref{c3}, and \eqref{c4} we have the following Theorem \ref{thm:expbvp}, whose proof we present in the forthcoming sections.
\begin{theorem}\label{thm:expbvp}
There exist { a sequence $\{(\hat{\mu}_j,\hat{c}_j)\}_{j\in \mathbb{N}}$ in $C^{m-1,\alpha}(\partial \Omega_Q,\mathbb{R}^n)_0\times \mathbb{R}^n$} and a real number $k_\# \in ]0,k_0[$ such that
\begin{equation}\label{eq:exbvp2}
 u[k](x)=\sum_{j=0}^{+\infty}v_{q}^-[\omega,\hat{\mu}_j](x)k^j+\frac{1}{k^{l}}\sum_{j=0}^{+\infty}\hat{c}_jk^j+Bq^{-1}x \qquad \forall x \in \overline{\mathbb{S}[\Omega_Q]^{-}}\, , \forall k \in ]0,k_\#[\, ,
\end{equation}
where  for all $k \in ]-k_\#,k_\#[$ the series $\sum_{j=0}^{+\infty}v_{q}^-[\omega,\hat{\mu}_j](x)k^j$ converges normally in $C^{m,\alpha}_q(\overline{\mathbb{S}[\Omega_Q]^{-}},\mathbb{R}^n)$ and $\sum_{j=0}^{+\infty}\hat{c}_jk^j$ converges normally in $\mathbb{R}^n$.
\end{theorem}
Equation  \eqref{eq:exbvp2} shows that $u[k]$ (which is defined only for $k \in ]0,k_0[$) can be represented in terms of power series which converge in a whole neighborhood of the degenerate value $k=0$. In Corollary \ref{cor:seq} we see how the sequence $\{(\hat{\mu}_j,\hat{c}_j)\}_{j\in \mathbb{N}}$ can be computed solving certain boundary integral equations.  {To conclude,} we observe that the Functional Analytic Approach of this paper has already been used for the analysis of perturbation problems for the Lam\'e equation in \cite{DaLa10a, DaLa10b} for bounded domains and in \cite{DaMu14, FaLuMu21} for periodic domains.

\section{Preliminaries of periodic potential theory for the Lam\'e equations}\label{notprel}

In order to construct the solution of problem \eqref{bvp}, we will exploit a periodic version of potential theory for the Lam\'e {equations. We say} that a function $f$ on $\overline{\mathbb{S}[\Omega_Q]^{-}}$ is $q$-periodic if $f(x+qz)=f(x)$ for all $x \in \overline{\mathbb{S}[\Omega_Q]^{-}}$ and $z \in \mathbb{Z}^n$. To construct periodic elastic layer potentials, we introduce a periodic analog of the fundamental solution of $L[\omega]$ (cf., {e.g.}, Ammari and Kang \cite[Lemma 9.21]{AmKa07}, \cite[Thm.~3.1]{DaMu14}). So let $\Gamma_{n,\omega}^{q}{:=} (\Gamma_{n,\omega,j}^{q,k})_{(j,k)\in\{1,\dots,n\}^2}$ be the $n\times n$ matrix of $q$-periodic distributions with $(j,k)$ entry defined by
\[
\Gamma_{n,\omega,j}^{q,k}{:=} \sum_{z \in \mathbb{Z}^n \setminus \{0\}} \frac{1}{4 \pi^2 |Q|  |q^{-1}z|^2}\Biggl[ -\delta_{j,k}+\frac{\omega}{\omega+1}\frac{(q^{-1}z)_j(q^{-1}z)_k}{|q^{-1}z|^2}\Biggr]E_{2 \pi iq^{-1}z} \qquad \forall (j,k) \in \{1,\dots,n\}^2\,,
\]
where $E_{2\pi i q^{-1} z}(x){:=} e^{2\pi i (q^{-1} z)\cdot x}$  for all $x\in{\mathbb{R}}^{n}$ and $z\in\mathbb{Z}^n$. Then {$
L[\omega] \Gamma_{n,\omega}^{q}=\sum_{z \in \mathbb{Z}^n}\delta_{qz}I_n-\frac{1}{ |Q|}I_n$, in} the sense of distributions, where $\delta_{qz}$ denotes the Dirac measure with mass at $qz$. We mention that similar constructions have been used to define a periodic analog of the fundamental solution for an elliptic differential operator  in \cite[Chapter 12]{DaLaMu21} and for the heat equation in Luzzini \cite{Lu20}. We  set $\Gamma_{n,\omega}^{q,j}{:=} \bigl(\Gamma_{n,\omega,i}^{q,j}\bigr)_{i \in \{1,\dots,n\}}$, which we think as column vectors for all $j\in\{1,\dots,n\}$.  We now  introduce the periodic single layer potential. So, if $\mu \in C^{0,\alpha}(\partial {\Omega_Q},\mathbb{R}^n)$, then we denote by $v_q[\omega, \mu]$ the periodic single layer potential, defined as
\[
v_q[\omega, \mu](x){:=} \int_{\partial {\Omega_Q}}\Gamma^q_{n,\omega}(x-y)\mu(y)\,d\sigma_y \qquad \forall x \in \mathbb{R}^n\,.
\]
 If $\mu\in C^{0,\alpha}(\partial{\Omega_Q},\mathbb{R}^n)$, then $v_{q}[\omega,\mu]$ is $q$-periodic,  $L[\omega]v_{q}[\omega,\mu]
=
-\frac{1}{|Q|}\int_{\partial{\Omega_Q}}\mu \,d\sigma$ in  $
 {\mathbb{R}}^{n}\setminus\partial{\mathbb{S}}[{\Omega_Q}]^-$. We set
\begin{align*}
&V_q[\omega, \mu](x){:=} v_q[\omega, \mu](x) \quad \forall x \in \partial \Omega_Q\,, \quad W_{q}^\ast[\omega, \mu](x){:=} \int_{\partial {\Omega_Q}}\sum_{l=1}^n \mu_{l}(y)T(\omega,D\Gamma_{n,\omega}^{q,l}(x-y))\nu_{{\Omega_Q}}(x)\,d\sigma_y \quad \forall x \in \partial {\Omega_Q}\,.
\end{align*}
 If $\mu\in C^{m-1,\alpha}(\partial{\Omega_Q},\mathbb{R}^n)$, then
$v^{-}_{q}[\omega,\mu]{:=} v_{q}[\omega,\mu]_{|\overline{\mathbb{S}[\Omega_Q]^{-}}}$ belongs to the Schauder space of $q$-periodic functions $C^{m,\alpha}_{q}
(\overline{\mathbb{S}[\Omega_Q]^{-}},\mathbb{R}^n)$ (equipped with its usual norm) and the operator $\mu \mapsto v^{-}_{q}[\omega,\mu]$   is  continuous from  $C^{m-1,\alpha}(\partial{\Omega_Q},\mathbb{R}^n)$ to $C^{m,\alpha}_{q}
(\overline{\mathbb{S}[\Omega_Q]^{-}},\mathbb{R}^n)$. Moreover, the operator $\mu \mapsto W_{q}^\ast[\omega,\mu]$ is  continuous from   $C^{m-1,\alpha}(\partial{\Omega_Q},\mathbb{R}^n)$ to itself, and we have $
T\bigl(\omega,Dv_{q}^{-}[\omega,\mu](x)\bigr)\nu_{{\Omega_Q}}(x)=\frac{1}{2}\mu(x)+W_{q}^\ast[\omega,\mu](x)$ for all $x \in \partial{\Omega_Q}$ and all $\mu \in C^{m-1,\alpha}(\partial{\Omega_Q},\mathbb{R}^n)$.

\section{Integral equation formulation of (\ref{bvp}) and proof of Theorem \ref{thm:expbvp}}
\label{inteqbvp}
First of all, by exploiting the periodic elastic single layer potential and \cite[Thm.~4.4]{DaMiMu22} on the representation of the solution of a Robin-type traction problem in a periodic domain, we immediately deduce the validity of the following proposition where we convert problem \eqref{bvp} into an integral equation.
\begin{proposition}\label{prop:exbvp}
Let $k\in ]0,k_0[$. Let $C^{m-1,\alpha}(\partial {\Omega_Q},\mathbb{R}^n)_0{:=} \left\{f \in C^{m-1,\alpha}(\partial {\Omega_Q},\mathbb{R}^n)\colon \int_{\partial {\Omega_Q}}f \, d\sigma=0\right\}$. Then
\[
 u[k](x)=v_{q}^-[\omega,{\mu_k}](x)+\frac{c_k}{k^{l}}+Bq^{-1}x \qquad \forall x \in \overline{\mathbb{S}[\Omega_Q]^{-}}\, ,
\]
where {$(\mu_k,c_k)$} is the unique {solution}   in $C^{m-1,\alpha}(\partial\Omega_Q,\mathbb{R}^n)_0\times \mathbb{R}^n$ {of}
\begin{align}\label{eq:exbvp3}
&\frac{1}{2}\mu(x) +W_{q}^\ast[\omega,\mu](x)+b[k](x)\Big(V_{q}[\omega, \mu](x)+\frac{c}{k^{l}}\Big)\\\nonumber &\qquad\qquad\qquad\qquad=g(x)-T(\omega,Bq^{-1})\nu_{\Omega_Q}(x)- { b[k](x)}Bq^{-1}x\qquad \forall x \in \partial \Omega_Q\, .
\end{align}
\end{proposition}
We introduce the operator $\Lambda$ from $]-k_0,k_0[$ to $\mathcal{L}(C^{m-1,\alpha}(\partial \Omega_Q,\mathbb{R}^n)_0\times \mathbb{R}^n, C^{m-1,\alpha}(\partial \Omega_Q,\mathbb{R}^n))$ defined by
\[
\Lambda[k](\mu,c)(x) :=\frac{1}{2}\mu(x) +W_{q}^\ast[\omega,\mu](x)+b[k](x)V_{q}[\omega, \mu](x) +k^{-l}b[k](x) c  \qquad \forall x \in \partial \Omega_Q\, ,
\]
for all $(\mu,c) \in C^{m-1,\alpha}(\partial \Omega_Q,\mathbb{R}^n)_0\times \mathbb{R}^n$ and for all $k \in ]-k_0,k_0[$.
We observe that \eqref{eq:exbvp3} can be rewritten as
\[
\begin{split}
\Lambda[k](\mu,c)(x)=g(x)-T(\omega,Bq^{-1})\nu_{\Omega_Q}(x)-b[k](x)Bq^{-1}x\qquad \forall x \in \partial \Omega_Q\, .
\end{split}
\]
Moreover, for $k=0$ the linear operator $\Lambda[0]$ becomes
\begin{equation}\label{eq:Lambda0}
\begin{split}
\Lambda[0](\mu,c)(x){=}&\frac{1}{2}\mu(x) +W_{q}^\ast[\omega,\mu](x) +\tilde b(x) c \qquad \forall x \in \partial \Omega_Q\, , \forall (\mu,c) \in C^{m-1,\alpha}(\partial \Omega_Q,\mathbb{R}^n)_0\times \mathbb{R}^n\, ,
\end{split}
\end{equation}
and $\Lambda[0]$ is invertible with bounded inverse in  $\mathcal{L}(C^{m-1,\alpha}(\partial \Omega_Q,\mathbb{R}^n), C^{m-1,\alpha}(\partial \Omega_Q,\mathbb{R}^n)_0\times \mathbb{R}^n)$ (see \cite[Lem.~4.2]{DaMiMu22}). Further properties of $\Lambda[k]$ are presented in the following.
\begin{proposition}\label{prop:anLmbd}
The following statements hold.
\begin{itemize}
\item[(i)] The map from $]-k_0,k_0[$ to $\mathcal{L}(C^{m-1,\alpha}(\partial \Omega_Q,\mathbb{R}^n)_0\times \mathbb{R}^n, C^{m-1,\alpha}(\partial \Omega_Q,\mathbb{R}^n))$  that takes $k$ to $\Lambda[k]$ is real analytic.
\item[(ii)] There exists $k_1 \in ]0,k_0[$ such that for each $k \in ]-k_1,k_1[$ the linear operator $\Lambda[k]$ is invertible with inverse in the space  $\mathcal{L}(C^{m-1,\alpha}(\partial \Omega_Q,\mathbb{R}^n), C^{m-1,\alpha}(\partial \Omega_Q,\mathbb{R}^n)_0\times \mathbb{R}^n)$ and such that the map from $]-k_1,k_1[$ to $\mathcal{L}(C^{m-1,\alpha}(\partial \Omega_Q,\mathbb{R}^n), C^{m-1,\alpha}(\partial \Omega_Q,\mathbb{R}^n)_0\times \mathbb{R}^n)$ that takes $k$ to $(\Lambda[k])^{(-1)}$ is real analytic.
\end{itemize}
\end{proposition}
\begin{proof}
The validity of  (i) follows by the boundedness of the linear operators $W_{q}^\ast[\omega,\cdot]$ and $V_{q}[\omega,\cdot]$ and by the real analyticity of $k\mapsto k^{-l}b[k]$. To prove (ii), we note that since the set of linear
homeomorphisms is open in the set of linear and continuous operators,
and since the map that takes a linear invertible operator to its
inverse is real analytic (cf.~\textit{e.g.}, Hille and
Phillips~\cite[Thms.~4.3.2 and 4.3.4]{HiPh57}), there exists $k_1 \in ]0,k_0[$ such that the map
that takes $k$ to $\Lambda[k]^{(-1)}$ is real analytic from $]-k_1,k_1[$ to $\mathcal{L}(C^{m-1,\alpha}(\partial \Omega_Q,\mathbb{R}^n), C^{m-1,\alpha}(\partial \Omega_Q,\mathbb{R}^n)_0\times \mathbb{R}^n)$.
\end{proof}

By Proposition \ref{prop:anLmbd}, we represent the solutions of the integral equation \eqref{eq:exbvp3} by means of real analytic maps.
\begin{corollary}
Let $(\hat{\mu},\hat{c})$ be the real analytic map from $]-k_1,k_1[$ to $C^{m-1,\alpha}(\partial \Omega_Q,\mathbb{R}^n)_0\times \mathbb{R}^n$ defined by $(\hat{\mu}[k],\hat{c}[k]) :=(\Lambda[k])^{(-1)} {\mathfrak{D}[k]}$  for all $k \in ]-k_1,k_1[$, where
\[
\mathfrak{D}[k](x):= g(x)-T(\omega,Bq^{-1})\nu_{\Omega_Q}(x)-b[k](x)Bq^{-1}x \qquad \forall k\in]-k_0,k_0[\,,\; x \in \partial \Omega_Q\, .
\]
Then $(\hat{\mu}[k],\hat{c}[k])=(\mu_k,c_k)$  for all $k \in ]0,k_1[$ and  $(\hat{\mu}[0],\hat{c}[0])$ is the unique solution $(\mu,c)$ in $C^{m-1,\alpha}(\partial \Omega_Q,\mathbb{R}^n)_0\times \mathbb{R}^n$ of
\begin{align}\label{eq:lim}
\frac{1}{2}\mu(x) &+W_{q}^\ast[\omega,\mu](x)+\tilde{b}(x) c=\mathfrak{D}[0](x)\qquad \forall x \in \partial \Omega_Q\, .
\end{align}
\end{corollary}
\begin{proof}
By Proposition \ref{prop:anLmbd},  $k \mapsto (\hat{\mu}[k],\hat{c}[k])$ is real analytic from $]-k_1,k_1[$ to $C^{m-1,\alpha}(\partial \Omega_Q,\mathbb{R}^n)_0\times \mathbb{R}^n$ and $(\hat{\mu}[k],\hat{c}[k])=(\mu_k,c_k)$  for all $k \in ]0,k_1[$. Since $(\hat{\mu}[0],\hat{c}[0]){:=} (\Lambda[0])^{(-1)} \mathfrak{D}{[0]}$, by equation \eqref{eq:Lambda0}, we deduce that $(\hat{\mu}[0],\hat{c}[0])$ is the unique solution $(\mu,c)$ in $C^{m-1,\alpha}(\partial \Omega_Q,\mathbb{R}^n)_0\times \mathbb{R}^n$ of
equation \eqref{eq:lim}.
\end{proof}
Since $k\mapsto \frac{{b[k]}}{k^l}$ is real analytic, there exist $\tilde{k}\in ]-k_0,k_0[$ and a family {$\{{b^\#_j}\}_{j\in \mathbb{N}}$ in $C^{m-1,\alpha}(\partial \Omega_Q,M_n(\mathbb{R}))$} such that
${b[k]}=k^l\sum_{j=0}^{+\infty}{b^\#_j}k^j$ for all $k \in ]-\tilde{k},\tilde{k}[$,  where the series $\sum_{j=0}^{+\infty}{b^\#_j}k^j$ converges normally in {$C^{m-1,\alpha}(\partial \Omega_Q,M_n(\mathbb{R}))$} for all $k \in ]-\tilde{k},\tilde{k}[$.
Possibly taking a smaller $\tilde{k}$, we note that
\begin{align*}
\Lambda[k](\mu,c)
&=\Lambda[0](\mu,c)+\sum_{j=1}^{+\infty}\Bigg({b^\#_{j-l}} V_{q}[\omega, \mu](x) +  {b^\#_j}   c\Bigg)k^j\, ,
\end{align*}
where we understand that ${b^\#_{j-l}}=0$ if $j<l$ and where the series $\sum_{j=1}^{+\infty}\Bigg({b^\#_{j-l}} V_{q}[\omega, \mu](x) +  {b^\#_j}   c\Bigg)k^j$ converges normally in $\mathcal{L}(C^{m-1,\alpha}(\partial \Omega_Q,\mathbb{R}^n)_0\times \mathbb{R}^n, C^{m-1,\alpha}(\partial \Omega_Q,\mathbb{R}^n))$ for all $k \in ]-\tilde{k},\tilde{k}[$.
We find convenient to set
\begin{align*}
&R_j(\mu,c){:=} {b^\#_{j-l}} V_{q}[\omega, \mu](x) +  {b^\#_j}   c\qquad \forall j \in \mathbb{N}\setminus \{0\}\, ,\qquad R[k](\mu,c){:=} \sum_{j=1}^{+\infty}R_j(\mu,c)k^j\, ,
\end{align*}
and accordingly $\Lambda[k]=\Lambda[0]+R[k]$. By the Neumann series theorem, possibly taking again a smaller $\tilde{k}$, we have
\[
(\Lambda[k])^{(-1)}=(\Lambda[0])^{(-1)}+\sum_{r=1}^{+\infty}(-1)^r \bigg((\Lambda[0])^{(-1)}R[k]\bigg)^r (\Lambda[0])^{(-1)}\, ,
\]
where for all $k\in ]-\tilde{k},\tilde{k}[$ the series converges normally in $\mathcal{L}(C^{m-1,\alpha}(\partial \Omega_Q,\mathbb{R}^n),C^{m-1,\alpha}(\partial \Omega_Q,\mathbb{R}^n)_0\times \mathbb{R}^n)$. For all $r\in \mathbb{N}\setminus \{0\}$, we have
\[
\bigg((\Lambda[0])^{(-1)}R[k]\bigg)^r = \sum_{j=1}^{+\infty}\Bigg(\sum_{\substack{j_{l_1},\dots,j_{l_r}\in \mathbb{N}\setminus\{0\} \\ j_{l_1}+\dots j_{l_r}=j}}\bigg((\Lambda[0])^{(-1)}R_{j_1} \bigg)\cdot \dots \cdot \bigg((\Lambda[0])^{(-1)}R_{j_r} \bigg) \Bigg) k^j\, ,
\]
where for all $k\in ]-\tilde{k},\tilde{k}[$ the series converges normally in $\mathcal{L}(C^{m-1,\alpha}(\partial \Omega_Q,\mathbb{R}^n),C^{m-1,\alpha}(\partial \Omega_Q,\mathbb{R}^n)_0\times \mathbb{R}^n)$. Then {we set $L_0{:=} (\Lambda[0])^{(-1)}$ and for each $j\in \mathbb{N}\setminus \{0\}$} we define $L_j \in \mathcal{L}(C^{m-1,\alpha}(\partial \Omega_Q,\mathbb{R}^n),C^{m-1,\alpha}(\partial \Omega_Q,\mathbb{R}^n)_0\times \mathbb{R}^n)$ as
\begin{equation}\label{eq:Lj}
L_j{:=} \sum_{r=1}^{+\infty}(-1)^r \Bigg(\sum_{\substack{j_{l_1},\dots,j_{l_r}\in \mathbb{N}\setminus\{0\} \\ j_{l_1}+\dots +j_{l_r}=j}}\bigg((\Lambda[0])^{(-1)}R_{j_1} \bigg)\cdot \dots \cdot \bigg((\Lambda[0])^{(-1)}R_{j_r} \bigg) \Bigg)(\Lambda[0])^{(-1)}\, .
\end{equation}
Accordingly, possibly taking a smaller $\tilde{k}$, one can verify that {
$
(\Lambda[k])^{(-1)}=(\Lambda[0])^{(-1)}+\sum_{j=1}^{+\infty}L_j k^j\,
$, 
where for all $k\in ]-\tilde{k},\tilde{k}[$ the series converges normally in $\mathcal{L}(C^{m-1,\alpha}(\partial \Omega_Q,\mathbb{R}^n),C^{m-1,\alpha}(\partial \Omega_Q,\mathbb{R}^n)_0\times \mathbb{R}^n)$.
{Then we introduce the sequence $\{d_j\}_{j\in \mathbb{N}}$ in $C^{m-1,\alpha}(\partial \Omega_Q,\mathbb{R}^n)$ by setting
\begin{align*}
 d_0(x){:=} g(x)-T(\omega,Bq^{-1})\nu_{\Omega_Q}(x) \quad   \forall x \in \partial \Omega_Q\, ,
\qquad d_j(x) :=-b^\#_{j-l}(x)Bq^{-1}x  \quad \forall x \in \partial \Omega_Q \, ,  \forall j \in \mathbb{N}\setminus \{0\}\, ,
\end{align*}
where we understand that $b^\#_{j-l}=0$ if $j<l$. Possibly shrinking $\tilde{k}$, we note that $\mathfrak{D}[k]=\sum_{j=0}^{+\infty}d_jk^j$, where  for all $k\in ]-\tilde{k},\tilde{k}[$ the series converges normally in $C^{m-1,\alpha}(\partial \Omega_Q,\mathbb{R}^n)$.}  Then by the real analyticity of $k\mapsto (\hat{\mu}[k],\hat{c}[k])$ and {the expressions for $(\Lambda[k])^{(-1)}$ and for $\mathfrak{D}[k]$}, we deduce  the following.}

\begin{corollary}\label{cor:seq}
Let
\[
(\hat{\mu}_0,\hat{c}_0)=(\Lambda[0])^{(-1)} ({d_0})\, ,\qquad {(\hat{\mu}_j,\hat{c}_j)=\sum_{\substack{j_1,j_2 \in \mathbb{N} \\j_1+j_2=j}}L_{j_1} (d_{j_2})} \qquad \forall j \in \mathbb{N}\setminus \{0\}\, ,
\]
where { $L_0{:=} (\Lambda[0])^{(-1)}$ and} $L_j$ is  as in \eqref{eq:Lj} for  $j \in \mathbb{N}\setminus \{0\}$. {Then there exists $k_2\in ]0,k_1[$  such that $(\hat{\mu}[k],\hat{c}[k])=\sum_{j=0}^{+\infty}(\hat{\mu}_j,\hat{c}_j) k^j$  for all $k \in ]-k_2,k_2[$, where the series converges} normally in $C^{m-1,\alpha}(\partial \Omega_Q,\mathbb{R}^n)_0\times \mathbb{R}^n$ for all $k \in ]-k_2,k_2[$.
\end{corollary}
%
We are now able to prove  Theorem \ref{thm:expbvp}.
\begin{proof}[Proof of Theorem \ref{thm:expbvp}]
We already know that  for all $k \in ]-k_2,k_2[$ the series $\sum_{j=0}^{+\infty}\hat{c}_jk^j$ converges normally in $\mathbb{R}^n$  and that  the series $\sum_{j=0}^{+\infty}\hat{\mu}_jk^j$ converges normally in $C^{m-1,\alpha}(\partial \Omega_Q,\mathbb{R}^n)_0$. Since $v_{q}^-[\omega,\cdot]$ is a bounded linear operator from $C^{m-1,\alpha}(\partial \Omega_Q,\mathbb{R}^n)_0$ to $C^{m,\alpha}_q(\overline{\mathbb{S}[\Omega_Q]^{-}},\mathbb{R}^n)$, we deduce that taking a sufficiently small $k_\# \in ]0,k_2[$ for all $k \in ]-k_\#,k_\#[$ the series $\sum_{j=0}^{+\infty}v_{q}^-[\omega,\hat{\mu}_j](x)k^j$ converges normally in $C^{m,\alpha}_q(\overline{\mathbb{S}[\Omega_Q]^{-}},\mathbb{R}^n)$. Then, the representation formula of Proposition \ref{prop:exbvp} completes the proof.
\end{proof}

\section{Conclusions}

We have used the Functional Analytic Approach to study the  Lam\'e equations in a periodic domain  with a Robin-type boundary condition that turns into a pure traction one. The change in the boundary condition is obtained multiplying the Dirichlet-type term by a $k$-dependent matrix function $b[k](\cdot)$ that vanishes for $k=0$. We have seen that for $k>0$ close to $0$ the solution can be written as the sum of two converging power series of  $k$, one being multiplied by by  the singular function $1/k^l$, and a linear function that takes care of the quasi-periodicity of the solution (and disappears for periodic solutions). The positive natural number $l$ depends on the vanishing order of the matrix $b[k]$ as $k$  tends to $0$.

\section*{Acknowledgement}
P.M.~and G.M.~acknowledge the support from EU through the H2020-MSCA-RISE-2020 project EffectFact,
Grant agreement ID: 101008140.   M.D.R. and P.M.~are   members of the Gruppo Nazionale per l'Analisi Matematica, la Probabilit\`a e le loro Applicazioni (GNAMPA) of the Istituto Nazionale di Alta Matematica (INdAM). G.M.~acknowledges also Ser Cymru Future Generation Industrial Fellowship number AU224 -- 80761 and thanks the Royal Society for the Wolfson Research Merit Award.

\end{document}